\documentclass[12pt, twoside, leqno]{article}

\usepackage{amsmath,amsthm}
\usepackage{amssymb}

\usepackage{enumerate}

\usepackage{graphicx}

\newtheorem{thm}{Theorem}[section]

\newtheorem{lem}[thm]{Lemma}
\newtheorem{prop}[thm]{Proposition}
\newtheorem{prob}[thm]{Problem}

\newtheorem{mainthm}[thm]{Main Theorem}

\theoremstyle{definition}
\newtheorem{defin}[thm]{Definition}

\newtheorem*{xrem}{Remark}

\numberwithin{equation}{section}

\newcommand{\length}{\operatorname{length}}

\begin{document}


\baselineskip=17pt


\title{Some remarks on Kuratowski partitions}

\author{Joanna Jureczko and Bogdan W\c{e}glorz\\
Institute of Mathematics\\ 
Cardinal Stefan Wyszy\'nski University in Warsaw\\
E-mail: j.jureczko@uksw.edu.pl\\
E-mail: b.weglorz@uksw.edu.pl}

\date{}

\maketitle


\renewcommand{\thefootnote}{}

\footnote{2010 \emph{Mathematics Subject Classification}: Primary 03E05; Secondary 54E52.}

\footnote{\emph{Key words and phrases}: Baire property, Kuratowski partition, precipitous ideal, Fr\'echet ideal.}

\renewcommand{\thefootnote}{\arabic{footnote}}
\setcounter{footnote}{0}


\begin{abstract}
We introduce the notion of $K$-ideals associated with Kuratowski partitions and we prove that each $\kappa$-complete ideal on a measurable cardinal $\kappa$ can be represented as a $K$-ideal. Moreover, we show some results concerning precipitous and Fr\'echet ideals.
\end{abstract}

\section{Introduction}

The concept emerged when attempting to solve the problem set by 
K. Kuratowski in \cite{KK} whether each function $f \colon X\to~Y$, from a completely metrizable space $X$ to a metrizable space $Y$, such that for each open $V \subset X$ the set $f^{-1}(V)$ has the Baire property, (i.e. it differs from an open set by a meager set) is continuous apart from a meager set

As shown by A. Emeryk, R. Frankiewicz and W. Kulpa in  \cite{EFK}  this problem is equivalent to the problem of the nonexistence of partitions of a metrizable space into meager sets with the property that each its subfamily has the Baire property. Such a partition is called a \textit{Kuratowski partition}, (see Section 2 for a formal definition). 
In \cite{EFK} the authors did not used the name "Kuratowski partition", but they iisted its defining properties.
It seems that for the first time this name was used in \cite{FS}. 

With any Kuratowski partition of a topological space it is associated an ideal called in this paper \textit{$K$-ideal}, (see Section 2 for the formal definition). It seemed that knowledge of such $K$-ideal will determine if it is a Kuratowski partition. Unfortunately it is not so because, as we will show, the structure of such an ideal can be almost arbitrary. For "decoding" a Kuratowski partition from a given $K$-ideal we need also full information about the space in which we consider such an ideal.  

In 1987 R. Frankiewicz and K. Kunen in \cite{FK} showed, by  forcing methods, that the existence of Kuratowski partitions is equiconsistent with the existence of precipitous ideals. 

In this paper we show that with some assumptions    
the completion of some metric Baire space with a Kuratowski partition does not have a Kuratowski partition, (Proposition 3.1). Moreover using only combinatorial methods we show that a $K$-ideal associated with a Kuratowski partition of some space need not be precipitous, (Theorem 3.3) and each $\kappa$-complete ideal a measurable cardinal $\kappa$ can be represented by some $K$-ideal, (Theorem 3.4). 

\section{Definitions and basic facts}
 
Let $X$ be a topological space and let $\kappa$ be a regular cardinal. Let $\mathcal{F}$ be a partition of $X$ into meager sets. We say that $\mathcal{F}$ is \textit{a Kuratowski partition} if $\bigcup \mathcal{F}'$ has the Baire property for each $\mathcal{F}' \subseteq \mathcal{F}$.
\\
We can and will assume that $\mathcal{F}$ is indexed by $\kappa$, i.e
$$\mathcal{F} = \{F_\alpha: \alpha < \kappa\}.$$
With $\mathcal{F}$ we associate an ideal 
$$I_\mathcal{F} = \{A \subset \kappa \colon \bigcup_{\alpha \in A} F_\alpha \textrm{ is meager}\}$$
which we call \textit{a $K$-ideal}.

As was defined above, a Kuratowski partition $\mathcal{F}$ of a topological space $X$ is indexed by ordinals, but a $K$-ideal associated with $\mathcal{F}$ is a $\kappa$-complete ideal on cardinals. Thus in some results of this paper the considerations are carried out on cardinals instead of topological spaces.

Let $S$ be a measurable space with a positive measure and let $I$ be the ideal of all sets of measure zero. An \textit{$I$-partition} of $S$ is a maximal family $W$ of subsets of $S$ of a positive measure such that $A \cap B \in I$ for all distinct $A, B \in W$.
An $I$-partition $W_1$ of $S$ is a \textit{refinement} of an $I$-partition $W_2$ of $S$, $W_1 \leq W_2$, if each $A \in W_1$ is a subset of some $B\in W_2$.
  
Let $I$ be a $\kappa$-complete ideal on $\kappa$ containing singletons. 
The ideal $I$ is \textit{precipitous} if whenever $S$ is a set of a positive measure and $\{W_n \colon n < \omega\}$ are $I-partitions$ of $S$ such that 
$W_0 \geq W_1\geq ... \geq W_n \geq ...$
then there exists a sequence of sets
$A_0 \supseteq A_1\supseteq ... \supseteq A_n \supseteq ...$
such that $A_n \in W_n$ for each $n$, and $\bigcap_{n=0}^{\infty} A_n$ is nonempty. (see also \cite[p. 438-439]{TJ}).

A \textit{Fr\'echet ideal} is an ideal of the form
$I =\{A \subset \kappa \colon |A| < \kappa\}$, for some cardinal $\kappa$.

By \cite[Lemma 35.9, p. 440]{TJ} we know that

\begin{lem}[\cite{TJ}] 
	Let $\kappa$ be a regular uncountable cardinal. The Fr\'echet ideal is not precipitous.
\end{lem}

If $\lambda$ is a cardinal, then let $B(\lambda)$ denotes a metric space $(D(\lambda))^\omega$, where $D(\lambda)$ is a discrete space of cardinality $\lambda$, (see e. g. \cite{FK}).
\\
In \cite{FK} it is proved the following result, (\cite[Theorem 2.1]{FK}).

\begin{thm}[\cite{FK}]
	Assume that $J$ is an $\omega_1$-complete ultrafilter on $\kappa$. Then $B(2^\kappa)$ has a Kuratowski partition of cardinality $\kappa$.
\end{thm}

Let $I^+ = \mathcal{P}(\kappa) \setminus I$.
Consider a set
$$X(I) = \{x \in (I^+)^{\omega} \colon \bigcap\{x(n) \colon n \in \omega\} \not = \emptyset \textrm{ and } \forall_{n\in \omega} \bigcap\{x(m) \colon m < n\} \in I^+\}.$$
As was pointed out in \cite{FK} the set $X(I)$ is considered as a subset of a complete metric space $(I^+)^\omega$, where $I^+$ is equipped with the discrete topology. 
In \cite{FK}, the following facts were proved, (see \cite[Proposition 3.1 and Theorem 3.2]{FK}).

\begin{prop}[\cite{FK}]
$X(I)$ is a Baire space iff $I$ is a precipitous ideal.
\end{prop}

\begin{thm}[\cite{FK}]
Let $I$ a precipitous ideal on some regular cardinal. Then there is a Kuratowski partition of the metric Baire space $X(I)$. 
\end{thm}

Let $\tau(X)$ be a topology, (i. e. the family of open sets) on a set $X$.
\\
A family $\mathcal{B} \subset \tau(X)$ is a base of $X$ if each $U \in \tau(X)$ is a union of some members of $\mathcal{B}$. The \textit{weight} of $X$  is denoted as follows
$$w(X) = \min \{|\mathcal{B}| \colon \mathcal{B} \textrm{ a base of } X\}+\omega.$$
\noindent
A family $\mathcal{V} \subset \tau(X)\setminus \{\emptyset\}$  is a $\pi$-base of $X$ if for each nonempty open set $U\in \tau(X)$ there is a set $V \in \mathcal{V}$ with $V \subset U$.  
 The $\pi$\textit{-weight} of $X$ is defined as follows
$$\pi w(X) = \min \{|\mathcal{V}| \colon \mathcal{V} \textrm{ a $\pi$-base of }X\}+\omega.$$
(For more information see e. g. \cite[p. 10 and 14]{RH} or \cite[p. 5]{JI}).
\\
Note that if $X$ is an infinite metrizable space then $w(X) = \pi w(X)$, (see \cite[Theorem 8.1, p. 32-33]{RH}).

A space $X$ is a \textit{\v{C}ech complete space} if $X$ is a dense $G_\delta$ subset of a compact space, (see \cite[p. 252]{RE}).
Each \v{C}ech complete space is  a Baire space. 

\begin{thm}[\cite{EFK1}]
Let $X$ be a \v{C}ech complete space such that $\pi w(X) \leq 2^\omega$.  Then a Kuratowski partition of $X$ does not exists.
\end{thm}

For a given metric space $Y$, by $\tilde{Y}$ we denote its completion (in the sense of \cite[Theorem 4.3.19, p. 340]{RE}).
\\

Other notations of this paper are standard  for this area and can be found in \cite{TJ} (infinite combinatorics), \cite{RE} and \cite{KK} (topology).

\section{Main results} 
   
\begin{prop}
	If $2^\omega = 2^{\omega_1}$, 
	then there exists a metric Baire space with a Kuratowski partition for which a completion does not have a Kuratowski partition.
\end{prop}

\begin{proof}
	Let $I$ be a precipitous ideal on $\omega_1$. Let $X(I)$ be as defined in Section 2. By Proposition 2.3 $X(I)$ is a Baire space. 
	By Theorem 2.4. this space has a Kuratowski partition. 
	Consider $\tilde{X}(I)$. By \cite[Theorem 4.3.19, p. 340]{RE},  we have that $w(X(I)) = w(\tilde{X}(I))$. By \cite[Theorem 8.1, p. 32-33]{RH} we have that 
	$w(\tilde{X}(I)) = \pi w(\tilde{X}(I))$. 
	Using $2^\omega = 2^{\omega_1}$ we see that
	$|\pi w (\tilde{X}(I))| \leq 2^\omega$.
	By Theorem 2.5 the partition $\mathcal{F}$ is not a Kuratowski partition of $\tilde{X}(I)$. 
\end{proof}	
\noindent
Notice that in Proposition 3.1 we need not assume that $X$ is a metric space because as has been shown in  \cite[Lemma 5 and Lemma 6]{FS} if  a Kuratowski partition exists for a Hausdorff Baire space, then there also exists for a metric space.

\begin{prop}
	Let $\kappa$ be a regular cardinal.
	Let $X$ be a space with a Kuratowski partition of cardinality $\kappa$ and let $\Pi$ be a family of all permutations of $\kappa$. Then the direct sum   $\oplus_{\pi \in \Pi} X_\pi$ has a Kuratowski partition.
\end{prop}

\begin{proof}
	Let $\mathcal{F} = \{F_\alpha \colon \alpha < \kappa\}$ be a Kuratowski partition of $X$.
	Consider the set $\Pi$ of all permutations of $\kappa$. Let $\{X_\pi \colon \pi \in \Pi\}$ be a set of spaces homeomorphic to $X$ indexed by elements of $\Pi$.
	Consider the direct sum $\oplus_{\pi \in \Pi} X_\pi$. Of course  each $X_\pi$ is open in $\oplus_{\pi \in \Pi} X_\pi$. For each $\pi \in \Pi$ let $\mathcal{F}_\pi$ be a partition of $X_\pi$ such that 
	$$\mathcal{F}_\pi = \{F_{\pi(\alpha)} \colon \exists_{\beta < \alpha}\  \beta = \pi(\alpha) \textrm{ and }F_\beta \in \mathcal{F}\}.$$ 
	Such a family is a Kuratowski partition of $X_\pi$ for all $\pi \in \Pi$.
	For each $\alpha < \kappa$ consider
	$$F^*(\alpha) = \bigcup\{F_{\pi(\alpha)} \colon \pi \in \Pi\}.$$
	Notice that $\mathcal{F}^* = \{F^*(\alpha) \colon \alpha < \kappa\}$ is a Kuratowski partition of  $\oplus_{\pi \in \Pi} X_\pi$. If not, then there exists a subfamily $\mathcal{F}^{*}_{0} \subseteq \mathcal{F}^*$ has not the Baire property. Then by \cite[Union Theorem, p. 82]{KK1} we have that at least one of the elements of $\mathcal{F}^*_{0}$ is not meager. A contradiction.   
\end{proof}

With the Kuratowski partition $\mathcal{F}^*$ of the direct sum $\oplus_{\pi \in \Pi} X_\pi$ of copies of $X$  considered in the proof of Proposition 3.2 we may assiociate a $K$-ideal $I_{\mathcal{F}^*}$. This $K-$ideal can be very small.
  
\begin{thm}
Let $\kappa$ be a regular cardinal.
Let $X$ be a space with a Kuratowski partition $\mathcal{F}$ of cardinality $\kappa$. Then a $K$-ideal $I_{\mathcal{F}^*}$ is a Fr\'echet ideal.	
\end{thm}

\begin{proof}
	Let $\mathcal{F} = \{F_\alpha \colon \alpha < \kappa\}$ be a Kuratowski partition of $X$ and let $I_{\mathcal{F}^*}$ be a $K$-ideal. Suppose that there exists $A \in I_{\mathcal{F}^*}$ of cardinality $\kappa$. Then there exists a permutation $\pi$ of $\kappa$ and a set $B \subseteq \kappa$ of cardinality $\kappa$ such that $\pi(A) = B$ and $\bigcup_{\alpha \in B} F_{\pi(\alpha)}$ is nonmeager. Hence by \cite[Union Theorem, p. 82]{KK1} one of the elements in $\{F_{\pi(\alpha)} \colon \alpha \in B\}$ is nonmeager. A contradiction. 
\end{proof}

By results in \cite{FK} one can suppose that large cardinals and $K$-ideals sre strongly related, but comparing Proposition 3.3 and Lemma 2.1 one can conclude that such a $K$-ideal can be almost arbitrary, not necessary precipitous.

By \cite[Theorem 3.3 and Theorem 3.4]{FK}, the existence of a Kuratowski partition on an arbitrary space is equiconsistent with the existence of a measurable cardinal. 
In the next theorem it  will be shown that  reducing an ideal up to the Fr\'echet ideal can  be obtained by enlarging the space, as a direct sum. The idea of getting a measurable cardinal is a result of so called localization property, (see \cite{B} and \cite{FK}). Summarize we have a following theorem.

\begin{thm}
	Let $\kappa$ be a measurable cardinal. Then each $\kappa$-complete ideal $I$ on $\kappa$ can be represented by some $K$-ideal, (i.e. for each $\kappa$-complete ideal $I$ on $\kappa$ there exists a space with a Kuratowski partition $\mathcal{F}$ of cardinality $\kappa$ such that $I$ is of the form $I_\mathcal{F}$). 
\end{thm}

\begin{proof}
	Since $\kappa$ is measurable there exists a $\kappa$-complete nonprincipal ultrafilter on $\kappa$
	$$J= \{P_\xi \colon \xi \in 2^\kappa\}.$$
	Consider a space $J^\omega$, where $J$ is equipped with a discrete topology.
	Since $J$ is $\kappa$-complete, $\min \bigcap \{P_{x(n)} \colon n \in \omega\}$ is nonempty.
	Define 
	$$F_\alpha = \{x \in J^\omega \colon \alpha = \min \bigcap \{P_{x(n)} \colon n \in \omega\}\}.$$
	Then the family $\mathcal{F} = \{F_\alpha \colon \alpha < \kappa\}$ is a Kuratowski partition, (see the proof of \cite[Theorem 2.1]{FK} for details).	
	Let 
	$$I = \{B_\xi \colon B_\xi = \kappa \setminus P_\xi \textrm{ for some } P_\xi \in J, \xi < 2^\kappa \}.$$
	Obviously $I$ is a $\kappa$-complete maximal ideal.
	Now we enlarge the space $J^\omega$ by some its copies, (i.e. spaces homeomorphic to $J^\omega$) by excluding from $I$ all subsets $B$ of cardinality $\kappa$ for which  $\bigcup_{\alpha \in B} F_\alpha$ is nonmeager.
	To do this we proceed as follows: for each set $B \in I$ of cardinality $\kappa$ such that $\bigcup_{\alpha \in B} F_\alpha$ is nonmeager take a permutation $\pi_B$ of $\kappa$ such that $\bigcup_{\beta \in \pi_B (B)} F_\beta$ is meager.
	For simplifying the notation take 
	$A_B = \bigcup_{\alpha \in B} F_\alpha$ and 
	$A_{\pi_B} = \bigcup_{\beta \in \pi_B (B)} F_\beta$. Let 
	$$\Pi(I) = \{\pi_B \in \kappa ! \colon \exists_{B \in I} |B| = \kappa\  (\textrm{if } A_B \textrm{ is meager, then } A_{\pi_B} \textrm{ is nonmeager)}\}.$$
	Consider $\oplus_{\pi_B \in \Pi(I)} A_{\pi_B}$.
	Since each $A_{\pi_B}$  is meager then also $\oplus_{\pi_B \in \Pi(I)} A_{\pi_B}$ is meager.
	Then the ideal 
	$$I' = \{B \in I \colon A_B \textrm{ is meager or } \oplus_{\pi_B \in \Pi(I)} A_{\pi_B} \textrm{ is meager}\}$$
	is a required $K$-ideal.
\end{proof} 

\noindent
\begin{xrem} 
	If $\kappa$ is nonmeasurable but there exists a Kuratowski partition of cardinality $\kappa$ of a space $X$ then one can obtain each $\kappa$-complete ideal $K$ such that $J \subseteq K \subseteq I$, where $J$ is a Fr\'echet ideal and $I$ is an ideal from Theorem 3.4. 
	Thus, the existence of a Kuratowski partition leads to the statement that there exists a precipitous ideal, (see \cite{FJ}).
\end{xrem}
\noindent
\textbf{Acknowledgements.} We are very grateful for Reviewers and Editors who had the big influence of the final version of this paper.

\end{document}